\documentclass[12pt,reqno]{amsart}

\usepackage{epsfig}
\usepackage{amsmath, amsthm, amsfonts}
\usepackage[all]{xy}
\usepackage{graphicx}
\usepackage{pgf}
\usepackage{hyperref}

\numberwithin{equation}{section}

\theoremstyle{plain}
\newtheorem*{thm*}{Theorem}

\newtheorem{prop}{Proposition}

\newcommand{\triple}{$\left(\sigma_0,\sigma_1,\sigma_\infty\right)$}
\DeclareRobustCommand{\Co3}{$\text{Co}_3$}

\begin{document}

\author{Hartmut Monien}

\address{
  Bethe Center\\
  University Bonn\\
  Nussallee 12\\
  53115 Bonn\\
  Germany
}

\email{hmonien@uni-bonn.de}

\date{\today}

\title{The sporadic group $\text{Co}_\mathbf{3}$\\Hauptmodul and Bely{\u\i} map}

\begin{abstract}
  We calculate the hauptmodul and Bely{\u\i} map of a
  genus zero subgroup of the modular group defined via a canonical
  homomorphism by the Conway group group $Co_3$. Our main result is
  the Bely{\u\i} map and its field of definition.
\end{abstract}

\maketitle 

The Conway group \Co3 is a sporadic simple group of order
$2^{10}\cdot3^7\cdot5^3\cdot7\cdot11\cdot23\approx5\times10^{11}$. Geometrically
it is the subgroup of the automorphism group of the even unimodular
24-dimensional Leech lattice $\Lambda_{24}$ stabilizing 276 lines in
the lattice \cite{MR1662447}. It is isomorphic to a (2, 3, 7)
permutation group of 276 objects \cite{MR648901} and is therefore a
quotient of the full modular group.  The Riemann surface $X(\Gamma)$
of the corresponding subgroup $\Gamma$ of the modular group has genus
zero \cite{MR1242832}.  By Bely{\u\i}'s theorem \cite{MR534593} the
map $\Phi:X(\Gamma)\rightarrow\mathbb{P}^1$ is in
$\overline{\mathbb{Q}}(x)$.  In this note we focus on one of the
twelve triples of the type
$\left(\sigma_0, \sigma_1, \sigma_\infty\right) \in C_{2B} \times
C_{3C}\times C_{7A}$ as defined in \cite{MR648901} with
$\text{Co}_3 =\left<\sigma_0, \sigma_1\right>$ where
$\sigma_0,\sigma_1\in S_{276}$ and
$\sigma_\infty=\left(\sigma_0\sigma_1\right)^{-1} $ are fixed
throughout the note. The conjugacy classes $C_{2B}$, $C_{3C}$ and
$C_{7A}$ are labeled according to the Atlas notation \cite{MR827219}.
We present a Bely{\u\i} map $\Phi$ of degree 276 with a monodromy
group isomorphic to \Co3 and proof the following theorem.

\begin{thm*}
  The triple $\left(\sigma_0, \sigma_1, \sigma_\infty\right)$ defines
  up to simultaneous conjugation a non-congruence subgroup $\Gamma$ of
  genus zero of the full modular group with a rational Bely{\u\i} map
  $\Phi:X(\Gamma)\rightarrow\mathbb{P}^1$ which obeys the equation
  $\Phi(z) = p_3(z)/p_c(z) = 1728 + p_2(z)/p_c(z)$ relating the
  branching at the elliptic points of order two, three and the cusps
  with the polynomial $p_2$, $p_3$ and $p_c$ given in the accompanying
  material. These polynomials are defined over the number field
  $L=\mathbb{Q}[b]/(b^{36} - 15 b^{35} + 105 b^{34} - 452 b^{33} +
  1321 b^{32} - 2696 b^{31} + 3634 b^{30} - 2077 b^{29} - 3717 b^{28}
  + 11765 b^{27} - 13336 b^{26} - 4257 b^{25} + 46791 b^{24} - 104102
  b^{23} + 156805 b^{22} - 191498 b^{21} + 200457 b^{20} - 170957
  b^{19} + 98979 b^{18} - 17978 b^{17} - 18499 b^{16} - 638 b^{15} +
  28239 b^{14} - 22998 b^{13} + 15 b^{12} + 2577 b^{11} + 19524 b^{10}
  - 38036 b^{9} + 35169 b^{8} - 19422 b^{7} + 6174 b^{6} - 736 b^{5} -
  40 b^{4} - 154 b^{3} + 144 b^{2} - 48 b + 6)$ with discriminant
  $2^{26} \cdot 3^{13} \cdot 5^{18} \cdot 7^{27}$.
  \label{thm:Belyi_36}
\end{thm*}

The proof is established by constructing a finite index subgroup
$\Gamma\subset\text{SL}_2\left(\mathbb{Z}\right)$ of the modular group
corresponding to the triple \triple. A triple with the required
properties can be found in \cite{MR648901} (using the notation of that
paper our triple is given as $(d, e, (d e)^{-1}$). The cycle structure
of the triple is $\left(1^{12} 2^{132},3^{92}, 1^37^{39}\right) $ and
uniquely identifies the conjugacy classes of the generators as
$\sigma_0\in C_{2B}$, $\sigma_1\in C_{3C}$ and
$\left(\sigma_0\sigma_1\right)^{-1} \in C_{7A}$.  We briefly summarize
the data of $\Gamma$ which can be derived from the triple.  The triple
\triple\ is legitimate in the sense of Theorem 3 of Atkin and
Swinnerton-Dyer \cite{MR0337781}. As a consequence much of the
structure of $\Gamma$ can be read off from the cycle structure of the
generators. The group $\Gamma$ is an index 276 group of the modular
group with no elliptic points of order three, twelve elliptic points
of order two, thirty nine cusps of width seven and three cusps of
width one. Using the Riemann-Hurwitz formula the genus of $X(\Gamma)$
is zero (Theorem A of Magaard \cite{MR1242832}).  By Theorem (3.1) of
Hsu \cite{MR1343700} $\Gamma$ is a non-congruence subgroup of the
modular group. 

The analytic structure of the Bely{\u\i} function is determined by the
data above and the relation of the branching at the elliptic points of
order two, three and the cusps
$\Phi(x) = {p_3(x)}/{p_c(x)}= 1728+p_2(x)/p_c(x)$ where
$p_3(x) = (x^{92}+a_{91} x^{91}+\ldots+a_0)^3$,
$p_c(x)=(x^2+b_1 x+b_0)(x^{39}+c_{38}x^{38}+\ldots+c_0)^7$ and
$p_2(x)=(x^{12}+d_{11}x^{11}+\ldots+d_0)(x^{132}+e_{131}x^{131}+\ldots+e_0)^2$
with unknown coefficients Let
$C = \left\{a_0,a_1\ldots a_{91}\right\} \cup \left\{b_0,b_1\right\}
\cup \left\{c_0,c_1\ldots c_{38}\right\} \cup \left\{d_0, d_1,\ldots
  d_{11}\right\} \cup \left\{e_0, e_1,\ldots e_{131}\right\}$ be the
set of coefficients with size $|C|=277$.  The coefficients of the
resulting polynomial $p_3(x)-p_2(x)-1728 p_c(x) = 0$ all have to
vanish resulting in 276 polynomial equations with integral
coefficients for the 277 undetermined coefficients for $C$. Fixing the
boundary conditions for the generator of the function field on
$X\left(\Gamma\right)$ yields an additional linear equation.  Let
$\mathcal{H} = \{x+iy\;| \;(x, y) \in \mathbb{R}^2, y>0\}$ be the
complex upper half plane and
$\mathcal{H^*} = \mathcal{H}\cup\mathbb{P}^1(\mathbb{Q})$ be the
extended complex upper half plane. A modular function is a complex
valued function that extends to a meromorphic function on the
compactified upper half plane $f:\mathcal{H}^*\rightarrow\mathbb{C}$
satisfying $f(\gamma z) =f(z)$ for every $\gamma\in\Gamma$ and every
$z\in\mathcal{H^*}$.  In the genus zero case the field of modular
functions is generated by one modular function usually called
hauptmodul. We choose one of the cusps with the smallest cusp width
(which is one here) as principal cusp and fix the hauptmodul
$j_\Gamma:\mathcal{H}^*\rightarrow\mathbb{C}$ uniquely by imposing a
growth condition on it. As $y\rightarrow\infty$ it grows like
$j_\Gamma(z=x+i y) = q^{-1} + 0 + O\left(q\right)$ where
$q = \exp(2\pi i z)$ and stays finite at all other cusps. The standard
choice for the hauptmodul or Klein invariant of the full modular group
is denoted by $j$ has the Fourier expansion
$j(z)=q^{-1}+744+196884 q+O\left(q^2\right)$. The Bely{\u\i} map is
uniquely defined by the condition $\Phi\left(j_\Gamma\right)=j$ and
the boundary condition at infinity provides the additional linear
equation $3 a_{91}-b_1-7c_{38}=744$.  We have determined a solution of
the 277 algebraic equations numerically over $\mathbb{C}$ with
$2^{20}=1,048,576$ binary digits precision and identified all of the
corresponding algebraic numbers using the {\em pari/gp} implementation
\cite{PARI2} of the LLL-algorithm \cite{MR682664}.  All coefficients
are contained in the number field
$L=\mathbb{Q}[b]/(b^{36} - 15 b^{35} + 105 b^{34} - 452 b^{33} + 1321
b^{32} - 2696 b^{31} + 3634 b^{30} - 2077 b^{29} - 3717 b^{28} + 11765
b^{27} - 13336 b^{26} - 4257 b^{25} + 46791 b^{24} - 104102 b^{23} +
156805 b^{22} - 191498 b^{21} + 200457 b^{20} - 170957 b^{19} + 98979
b^{18} - 17978 b^{17} - 18499 b^{16} - 638 b^{15} + 28239 b^{14} -
22998 b^{13} + 15 b^{12} + 2577 b^{11} + 19524 b^{10} - 38036 b^{9} +
35169 b^{8} - 19422 b^{7} + 6174 b^{6} - 736 b^{5} - 40 b^{4} - 154
b^{3} + 144 b^{2} - 48 b + 6)$. Now we are in the position to proof
the Theorem.

\begin{proof}
  The pair $\left(\sigma_0, \sigma_1\right)$ is a legitimate pair of
  Theorem 3 of \cite{MR0337781} it therefore defines a subgroup
  $\Gamma$ of the modular group. By Theorem (3.1) of Hsu
  \cite{MR1343700} $\Gamma$ is a non-congruence subgroup and by
  Theorem A of Magaard the genus of $X\left(\Gamma\right)$ is
  zero. The polynomials $p_3$, $p_2$ and $p_c$ given in the appendix
  are in the polynomial ring $L[x]$, have the factorization structure
  given above, $p_3(x)-p_2(x)-1728 p_c(x)$ is identically zero and the
  additional linear equation given is obeyed. The jacobian of the
  equations for the coefficients of the polynomials has full rank at
  the solution. Therefore the solution is unique and fixed by the
  choice of $j_\Gamma$.
\end{proof}

The number field $L$ has only one non-trival subfield $K$ of degree
twelve
$K=\mathbb{Q}[a]/(a^{12} - 2 a^{11} + 9 a^{10} - 20 a^{9} + 38 a^{8} -
73 a^{7} + 101 a^{6} - 86 a^{5} + 55 a^{4} - 46 a^{3} + 42 a^{2} - 24
a + 6)$ which raises immediately the question if the Bely{\u\i} map of the 
Theorem defined over $L(x)$ can be expressed in the
smaller subfield $K(x)$. Let $w\in L(x)$ be a Moebius transform given
in the accompanying material. We define the action of $w$ on an
element of $p\in L[x]$ as the elementwise action of $w$ on the
coefficients of $p$ and denote it by $p\rightarrow\tilde p$.

\begin{prop}
  The transformed polynomials $\tilde p_3$, $\tilde p_2$ and $\tilde
  p_c$ are in $K[x]$ and there exists nonzero coefficients $(k_2, k_3,
  k_c)\in K^3$ such that $k_3\tilde p_3 + k_2 \tilde p_2 + k_c \tilde
  p_c $ identically vanishes.
\end{prop}

\begin{proof}
  The proof is purely computational. A set of non-vanishing
  coefficients $k_2$, $k_3$ and $k_c$ is given in the accompanying
  material.
\end{proof}

\bibliographystyle{alpha}

\begin{thebibliography}{CCN{\etalchar{+}}85}

\bibitem[ASD71]{MR0337781}
A.~O.~L. Atkin and H.~P.~F. Swinnerton-Dyer.
\newblock Modular forms on noncongruence subgroups.
\newblock In {\em Combinatorics ({P}roc. {S}ympos. {P}ure {M}ath., {V}ol.
  {XIX}, {U}niv. {C}alifornia, {L}os {A}ngeles, {C}alif., 1968)}, pages 1--25.
  Amer. Math. Soc., Providence, R.I., 1971.

\bibitem[Bel79]{MR534593}
G.~V. Bely\u\i.
\newblock Galois extensions of a maximal cyclotomic field.
\newblock {\em Izv. Akad. Nauk SSSR Ser. Mat.}, 43(2):267--276, 479, 1979.

\bibitem[CCN{\etalchar{+}}85]{MR827219}
J.~H. Conway, R.~T. Curtis, S.~P. Norton, R.~A. Parker, and R.~A. Wilson.
\newblock {\em Atlas of finite groups}.
\newblock Oxford University Press, Eynsham, 1985.
\newblock Maximal subgroups and ordinary characters for simple groups, With
  computational assistance from J. G. Thackray.

\bibitem[CS99]{MR1662447}
J.~H. Conway and N.~J.~A. Sloane.
\newblock {\em Sphere packings, lattices and groups}, volume 290 of {\em
  Grundlehren der Mathematischen Wissenschaften [Fundamental Principles of
  Mathematical Sciences]}.
\newblock Springer-Verlag, New York, third edition, 1999.
\newblock With additional contributions by E. Bannai, R. E. Borcherds, J.
  Leech, S. P. Norton, A. M. Odlyzko, R. A. Parker, L. Queen and B. B. Venkov.

\bibitem[Hsu96]{MR1343700}
Tim Hsu.
\newblock Identifying congruence subgroups of the modular group.
\newblock {\em Proc. Amer. Math. Soc.}, 124(5):1351--1359, 1996.

\bibitem[LLL82]{MR682664}
A.~K. Lenstra, H.~W. Lenstra, Jr., and L.~Lov\'asz.
\newblock Factoring polynomials with rational coefficients.
\newblock {\em Math. Ann.}, 261(4):515--534, 1982.

\bibitem[Mag93]{MR1242832}
Kay Magaard.
\newblock Monodromy and sporadic groups.
\newblock {\em Comm. Algebra}, 21(12):4271--4297, 1993.

\bibitem[PAR18]{PARI2}
{The PARI~Group}, Univ. Bordeaux.
\newblock {\em {PARI/GP version {\tt 2.9.4}}}, 2018.
\newblock available from \url{http://pari.math.u-bordeaux.fr/}.

\bibitem[Wor82]{MR648901}
M.~F. Worboys.
\newblock Generators for the sporadic group {${\rm Co}_{3}$}\ as a
  {$(2,\,3,\,7)$}\ group.
\newblock {\em Proc. Edinburgh Math. Soc. (2)}, 25(1):65--68, 1982.

\end{thebibliography}
\newcommand{\etalchar}[1]{$^{#1}$}

\end{document}